\documentclass[12pt]{article}
\usepackage{amsmath,amssymb}
\usepackage{amsthm}

\def\LC{\mathcal{L}}

\def\CC{\mathcal{C}}

\def\R{\mathbf{R}}

\def\1{\mathbf{1}}

\def\al{\alpha}

\def\pa{\partial}
\def\ep{\epsilon}
\def\de{\delta}

\def\ka{\varkappa}

\newtheorem{prop}{Proposition}[section]
\newtheorem{theorem}{Theorem}[section]

\newtheorem{remark}{Remark}

\newcommand{\si}{\sigma}

\newcommand{\om}{\omega}

\begin{document}
\title{Nonlinear diffusions and stable-like processes with coefficients depending on the median or VaR
\thanks{Supported by the AFOSR grant FA9550-09-1-0664 'Nonlinear Markov control processes and games' and by IPI of the Russian Academy of Science, grants RFBR 11-01-12026 and 12-07-00115}}
\author{Vassili N. Kolokoltsov\thanks{Department of Statistics, University of Warwick,
 Coventry CV4 7AL UK,
  Email: v.kolokoltsov@warwick.ac.uk}}
\maketitle

\begin{abstract}
In June 2012 on a conference in Bielefeld, after the author made the presentation of his theory of nonlinear Markov processes,
Tom Kurtz asked him whether his methods would allow to get well-posedness for nonlinear McKean-Vlasov type diffusions with coefficients depending on the distribution via the median. The present paper is the answer to this question. The aim here is not to find the weakest possible conditions for the affirmative answer, but a sensible general result with a quick and
transparent proof.
\end{abstract}

{\bf Key words:} McKean-Vlasov diffusion, nonlinear Markov processes, stable-like processes, tempered stable processes, median, quantile, value at risk (VaR).


\section{Main results}

By $\nabla$ we denote the gradient with respect to the spatial variable $x \in \R^d$, by $\|u\|_L$
the norm of functions in $L_1(\R^d)$, by $\de_x$ the Dirac measure at $x\in R^d$, by $|v|$ the Euclidean norm of a vector $v\in \R^d$.

For a positive integrable function $u$ on $\R$ and an $\al \in (0, \|u\|_L)$ the quantile $Q_{\al}(u)$ is defined as the number
satisfying
\[
\int_{-\infty}^{Q_{\al}(u)} u(x) \, dx =\al.
\]
If $\al=1/2$, the value  $Q_{\al}(u)$ is called the median of $u$. In finances, the values of $Q_{\al}(u)$ for small $\al$ specify the so-called Value at Risk (VaR). Other particular cases include quartiles and sextiles.

More generally, for any positive function $u$ on $\R^d$ with $\|u\|_L =1$ and any vector
 $\al=(\al_1,\cdots, \al_d)$ with all $\al_j\in (0,1)$, let us define the quantile
 vector $Q_{\al}(u)=(Q_{\al}^1,\cdots ,Q_{\al}^d)(u)$ with the coordinates uniquely specified by the equations
 \begin{equation}
\label{eqdefddimquant}
 \int_{ \{ x \in \R^d: \, x_j \le Q_{\al}^j(u)\}} u(x) \, dx =\al_j, \quad j=1, \cdots, d.
 \end{equation}

The problem addressed in this paper concerns the well-posedness of the nonlinear SDE
\begin{equation}
\label{eqSDEmedonedim}
dX_t=b(t,X_t, Q_{\al}(\LC (X_t))) \, dt + \si (t,X_t, Q_{\al}(\LC (X_t))) \, dW_t,
\end{equation}
where $W_t$ is a standard $d$-dimensional
 Brownian motion, $b,\si$ are, respectively, vector-valued and ($d\times d$ square symmetric)-matrix-valued continuous functions on $\R_+\times \R^d\times \R^d$,
and $\LC (X)$ denotes the distribution density of a random variable $X$. Such SDE arise in the analysis of interacting particle systems
related to asset pricing evaluation, see \cite{CKL}.

As an auxiliary problem we shall use the equation
\begin{equation}
\label{eqSDEmedprelim}
dX_t=b(t,X_t, \om (t)) \, dt + \si (t,X_t, \om (t)) \, dW_t,
\end{equation}
with a given continuous curve $\om (t)$ in $\R^d$. By Ito's formula, the corresponding diffusion has the generator
\begin{equation}
\label{eqSDEmedonedimprelimgen}
L_t= \frac12 \si^2 (t,x, \om (t)) \frac{\pa ^2}{\pa x^2}+b(t,x, \om (t)) \frac{\pa}{\pa x}
\end{equation}
in one-dimensional case and
 \begin{equation}
\label{eqSDEmedprelimgen}
L_t= \frac12 \sum_{i,j,k}\si_{ij}\si_{kj} (t,x, \om (t)) \frac{\pa ^2}{\pa x_i \pa x_k}+\sum_jb_j(t,x, \om (t)) \frac{\pa}{\pa x_j},
\end{equation}
or, in vector notations,
 \begin{equation}
\label{eqSDEmedprelimgen}
L_t= \frac12 (\si^2 (t,x, \om (t)) \nabla,\nabla)+(b(t,x, \om (t)), \nabla)
\end{equation}
in case of arbitrary $d$.

 Recall that a square matrix-valued function $a$ is called uniformly elliptic if
\begin{equation}
\label{eqdefunifelliptic}
m^{-1}|\xi|^2  \le (a(t,x)\xi, \xi) \le m |\xi|^2
\end{equation}
for all $\xi,x \in \R^d, t>0$ and some constant $m >0$.

It is well known (and follows from the existence of a regular Green function, see below) that, for uniformly elliptic coefficients, the distributions of solutions to equation \eqref{eqSDEmedprelim} have
 densities $u_t$ with respect to Lebesgue measure (even if the initial distribution does not have a density), and
the evolution of densities is known to satisfy the dual equation
\begin{equation}
\label{eqSDEmedprelimdual}
\frac{\pa }{\pa t}u_t(x)=L_t^*u(x)
= \frac12 \nabla ^2 [\si^2 (t,x, \om (t)) u_t(x)]
-\nabla (b(t,x, \om (t)), u_t(x)).
\end{equation}

The equation for the distribution densities of processes solving \eqref{eqSDEmedonedim} on $[0,T]$ becomes
\begin{equation}
\label{eqSDEmeddensity}
\frac{\pa }{\pa t}u_t(x)
= \frac12 \nabla ^2 [\si^2 (t,x, Q_{\al}(u_t)) u_t(x)]
-\nabla (b(t,x, Q_{\al}(u_t)), u_t(x)), \quad t\in (0,T].
\end{equation}

The well-posedness for this functional PDE is equivalent to the well-posedness of nonlinear SDE \eqref{eqSDEmedonedim}
 (In probabilistic language, solutions to \eqref{eqSDEmeddensity} supply weak solutions to SDE \eqref{eqSDEmedonedim}, but once
 the unique $Q_{\al}(u_t)$ is plugged into \eqref{eqSDEmedonedim}, its strong, path-wise, well-posedness becomes straightforward).

 It will be convenient to distinguish two cases, with $\si$ depending and not depending on $\om$, the second case requiring stronger regularity assumptions on the initial values $u_0$.

 \begin{theorem}
\label{thmediannonlin1}
Suppose $\si^2$ is uniformly elliptic, twice continuously differentiable in $x$ (with bounded derivatives)
and does not depend explicitly on $\om$,
$b$ is a bounded function continuously differentiable in $x$ (with bounded derivative), $\si,b$ are continuous functions of $t$
and $b$ is Lipschitz continuous function of $\om$, with the Lipschitz constant $\ka$.
Then, for any $T>0$ and any (strictly) positive $u_0 \in C(\R^d)\cap L_1(\R^d)$ there exists a unique classical solution $u_t \in C(\R^d)\cap L_1(\R^d)$ of equation \eqref{eqSDEmeddensity} with initial condition $u_0$, that is, it is continuous in $t\in [0,T]$ in the topology of $L_1(\R^d)$, satisfies \eqref{eqSDEmeddensity} for $t\in (0,T]$ and coincides with $u_0$ for $t=0$.
\end{theorem}

For the second result we need Sobolev spaces $H^k_1(\R^d)$ defined as the spaces of integrable functions on $\R^d$ with all derivatives up to and including order $k$ being well defined in distribution sense and being again integrable functions with the norm
$\|u\|_{H^k_1(\R^d)}$ defined as the sum of $L_1$-norms of $u$ and all its partial derivatives up to and including order $k$.

 \begin{theorem}
\label{thmediannonlin2}
Assume the same condition as above, but with $\si$ being also a Lipschitz continuous function of $\om$ (with the Lipschitz constant $\ka$). Assume also that $\si,b$ are twice continuously differentiable in $x$ (with bounded derivatives).
Then, for any $T>0$ and any (strictly) positive $u_0 \in H^2_1(\R^d)\cap C(\R^d)$, there exists a unique classical solution $u_t \in C(\R^d)\cap H^2_1(\R^d)$ of equation \eqref{eqSDEmeddensity} with initial condition $u_0$.
\end{theorem}

 \begin{theorem}
\label{thmediannonlin3}
Under the conditions of either of Theorems \ref{thmediannonlin1} or \ref{thmediannonlin2}, for any $\xi \in \R^d$,
there exists a unique global (i.e. defined for all positive times) strong solution of SDE \eqref{eqSDEmedonedim} with the initial condition $X_0=\xi$.
\end{theorem}

Our method is robust in the sense that it can be applied to other processes, not necessarily diffusions. Let us illustrate
this claim by applying it to stable-like processes. For simplicity we shall discuss the processes with a fixed stability
index $\al \in (1,2)$ and only with drift depending on a quantile.

Let us work directly with PDEs and generators. The corresponding stable SDEs can be written using the technique from \cite{BC}.
Let us consider the processes in $\R^d$ governed by the time-dependent generator
 \begin{equation}
\label{eqstablemedprelimgen}
L_t = (b(t,x, \om (t)), \nabla) - a(x)|\Delta|^{\al/2} .
\end{equation}
The evolution of densities of the corresponding process are generated by the dual operator, that is they solve the equation
\begin{equation}
\label{eqstablemedprelimdensity}
\frac{\pa }{\pa t}u_t(x)=L_t^*u(x)
= |\Delta|^{\al/2} [a(x) u(x)] -\nabla [b(t,x, \om (t)) u_t(x)].
\end{equation}

Our goal is the evolution depending on the quantile, that is, the equation

\begin{equation}
\label{eqstablemeddensity}
\frac{\pa }{\pa t}u_t(x)= |\Delta|^{\al/2} [a(x) u(x)] -\nabla [b(t,x,  Q_{\al}(u_t)) u_t(x)].
\end{equation}

 \begin{theorem}
\label{thmediannonlinstable}
Assume that $a(x)$ is a positive function, bounded from below and above and twice continuously differentiable (with bounded derivatives), and the vector-valued function
$b$ is twice continuously differentiable in $x$ (with bounded derivatives), continuous in $t$ and Lipschitz continuous in $\om$.
Then, for any $T>0$, any (strictly) positive $u_0 \in L_1(\R^d)\cap C(\R^d)$ and any $\xi \in \R^d$, there exist unique solutions $u_t \in C(\R^d)\cap L_1(\R^d)$ of equation \eqref{eqSDEmeddensity} with initial conditions $u_0$ or $\de_{\xi}$.
\end{theorem}

\section{Discussion}

The regularity of initial conditions and coefficients in Theorem \ref{thmediannonlin2} can definitely be weakened by using more advanced regularizing properties of heat equations. For instance, using the smoothing property for diffusion propagators (resolving operators) in the form
\begin{equation}
\label{eqsmoothheatinSobolev}
\|U^{s,t} f\|_{H_1^2(\R^d)} \le C (t-s)^{-1/2} \|f\|_{H_1^1(\R^d)}
\end{equation}
that should hold under certain regularity assumptions for coefficients (though the author did not find a proper reference),
one can directly improve Theorem \ref{thmediannonlin2} by allowing $u_0$ to belong to $H^1_1(\R^d)$ (and not only $H^2_1(\R^d)$).

Theorem \ref{thmediannonlin3} can be extended straightforwardly to a weakly dense (in the set of probability measures) class of initial conditions
having the property that their quantile is a single point. Say, if $d=1$, on can take the initial distribution in the form
$\sum_{j=1}^n p_j \de_{\xi_j}$, as long as there exists $k$ such that
\[
\sum_{j=1}^{k-1} p_j \de_{\xi_j} <\al < \sum_{j=1}^k p_j \de_{\xi_j}.
\]
Then the quantile $Q_{\al}(u_t)$ will tend to $\xi_k$ as $t\to 0$.
An interesting question remains on what can be said about processes with initial
distributions with the median not uniquely defined, as say $(\de_a+\de_b)/2$ with any $a\neq b$.

Theorem \ref{thmediannonlinstable} can be extended in several directions. It has a straightforward extension to processes with
variable stability index $\al (x)$ with corresponding required properties of transition probabilities following by (straightforward time-dependent extension of) Theorem 5.1 of \cite{Ko00} (see also Theorem 7.5.1 from  \cite{KoMarbook}), as well as stable processes
with variable spectral measure and/or perturbed by compound Poisson processes (say, tempered stable processes)
with required properties of transition probabilities following by Theorem 4.1 of \cite{Ko00}.
By developing appropriate smoothing results for time dependent stable-like processes, various 'stable' analogs
of Theorem \ref{thmediannonlin2} can be obtained.

Our arguments are inspired by the approach developed systematically in \cite{KoNonlbook}, though technical details are problem-specific. The main impetus for this development was an attempt to get a general understanding of the appearance of nonlinear measure-valued evolutions as limits of Markov models of interacting particles (extending the well developed theory
arising from the standard McKean-Vlasov diffusions). Recently, there appeared many publications combining mean-field interacting particle systems with control, see \cite{AnDj}, \cite{GaGaBou}, \cite{GMS2010}, \cite{HCM3}, \cite{KoLiYa}, \cite{Ko12}
for various directions. From the results of the present paper one can expect an interesting
development of mean-field control theory with peculiar (but from practical point of view quite natural) dependence of control
on empirical measures via such characteristics as median or VaR.

The rest of the paper is devoted to the proof of Theorems 1.1-1.3.

\section{Two-sided estimates for heat kernels and their consequences}

Let us denote by $G_{\si}$ the heat kernel of the standard heat equation in $\R^d$ with a diffusion coefficient $\si>0$, that is
\begin{equation}
\label{eqfreeheatkern}
G_{\si}(t,x)=(2\pi t \si^2)^{-d/2} \exp \{ -\frac{x^2}{2t \si^2} \}.
\end{equation}
Recall that $\int G_{\si}(t,x) \, dx=1$ for all $\si,t>0$.

For the general heat equation
\begin{equation}
\label{eqdiffeq}
\frac{\pa u}{\pa t}=L_t u, \quad L_tu (x)=\frac12(a(t,x)\nabla, \nabla)u(x)+(b(t,x), \nabla u(x))+c(t,x)u(x).
\end{equation}
the following fact is well known (see e.g. \cite{PorEid84}, \cite{LiSe00} and references therein).

\begin{prop}
\label{proptwosidedheat}
(i) Suppose $a$ is uniformly elliptic, that is \eqref{eqdefunifelliptic} holds, and continuously differentiable in $x$ and $a,b,c$ are bounded and continuous with respect to all their variables, so that
\[
\sup_{t,x} \max(|\nabla a(t,x)|, |b(t,x)|, |c(t,x)|)\le M,
\]
then there exist constants $\si_i, C_i$, $i=1,2$, depending only on $m,M, T$ such that
the Green function (or heat kernel) of equation \eqref{eqdiffeq} (i.e. its solution $G(t,x,s,\xi)$ with the initial condition $\de_{\xi}$ at time $s$) is well defined and satisfies the following two sided bounds
\begin{equation}
\label{eqdiffeqtwosidedbound}
C_1 G_{\si_1}(t-s, x-\xi)
\le G(t,x, s, \xi) \le C_2 G_{\si_2}(t-s, x-\xi), \quad 0<s,t<T.
\end{equation}

(ii) Moreover, if $a$ is twice continuously differentiable in $x$ and $b,c$ are continuously differentiable (with al derivatives bounded),
then $G(t,x, s, \xi)$ is differentiable in $x$ and
\begin{equation}
\label{eqdiffeqheatkernelderiv}
\left|\frac{\pa}{\pa x} G(t,x, s, \xi)\right| \le Ct^{-1/2} G(t,x, s, \xi), \quad 0<s,t<T,
\end{equation}
where the constant $C$ depends only on the bounds for the derivatives and $m,M,T$.
\end{prop}

\begin{remark} The known proofs of two-sided estimates (especially the lower one) are rather technical. Thus it is useful to
note (especially for possible generalizations) that we need only a weaker local lower bound (which follows from small time and near diagonal asymptotics of the Green function, see e.g. \cite{Ko00book}). Namely, for Proposition \ref{proplocalizemed} (i) we need the l.h.s. inequality \eqref{eqdiffeqtwosidedbound} for
$|x-\xi| \le c\sqrt {t-s}$ for at least one $c>0$, and for Proposition \ref{proplocalizemed} (ii) we need it for
$|x-\xi| \le c\sqrt {t-s}$ for sufficiently large $c>0$.
\end{remark}

Notice that the Green function for equation \eqref{eqSDEmedprelimdual} with $L_t^*$ is obtained by changing places of $x,\xi$ and $t,s$ in the Green function of the corresponding equation with $L_t$, which is of form \eqref{eqdiffeq}.

Let us deduce two corollaries, one from the upper bound and another from the lower bound of inequality \eqref{eqdiffeqtwosidedbound}.
Let us denote by $U_{m,M,t}(u_0)$ the set of solutions $u_t(x)$, $t>0$, of the Cauchy problems of all equations \eqref{eqdiffeq} satisfying the conditions of Proposition \ref{proptwosidedheat} with given $m,M$ and a given initial condition $u_0$. Also set
\[
U_{m,M,(0,T]}(u_0)=\cup_{t\in (0,T]}U_{m,M,t}(u_0).
\]

\begin{prop}
\label{propcorolupperbound}
For any $u_0 \in L_1(\R^d)$ and $\ep>0$ there exists $K>0$ such that
\begin{equation}
\label{eq1propcorolupperbound}
\int_{|x|\ge K} |u(x) | \, dx  \le \ep
\end{equation}
for all $u\in U_{m,M,[0,T]}(u_0)$ (that is, the set of measures $\{u (x) dx\}$, $u\in U_{m,M,[0,T]}(u_0)$ is tight).
\end{prop}

\begin{proof}
Let $K$ be chosen in such a way that
\[
\int_{|x|\ge K} |u_0(x) | \, dx  \le \ep,
\]
and let $\tilde K$ be another positive constant.

Recall that the solution to the Cauchy problem for equation \eqref{eqdiffeq} can be written in terms of its Green function as
\[
u_t(x)=\int  G(t,x, 0, \xi) u_0(\xi) \, d\xi.
\]
From the upper bound in \eqref{eqdiffeqtwosidedbound} we get
\[
\int_{|x| > K+\tilde K} u_t(x) \, dx \le C_2 \int_{|x| > K+\tilde K} \int_{\R^d} G_{\si_2}(t, x-\xi) u_0 (\xi) \, d\xi dx
\]
\[
\le  C_2 \int_{|\xi| > K, y\in \R^d} G_{\si_2}(t, y) u_0 (\xi) \, d\xi dy
+  C_2 \int_{|y| > \tilde K, |\xi| \le K} G_{\si_2}(t, y) u_0 (\xi) \, d\xi dy,
\]
which (by changing the variable of integration in the second integral to $z=y/\sqrt t$) is estimated by
\[
C_2 \ep +\|u_0\|_L \int_{|z| > \tilde K/\sqrt T} (2\pi \si_2^2)^{-d/2} \exp \{ -\frac{x^2}{2 \si_2^2} \} \, dz.
\]
Clearly the r.h.s. can be made arbitrary small by choosing $\ep>0$ small enough and $\tilde K$ large enough.
\end{proof}

\begin{prop}
\label{propcorollowerbound}
For any strictly positive continuous function $u_0$ on $\R^d$ and any $K>0$ there exists $\de>0$ such that
\begin{equation}
\label{eq1propcorollowerbound}
u(x) \ge \de
\end{equation}
for all $u\in U_{m,M,[0,T]}(u_0)$ whenever $|x| \le K$. In other words
\[
\inf \{ u(x): |x| \le K, u\in U_{m,M,[0,T]}(u_0) \} >0.
\]
\end{prop}

\begin{proof}
From positivity and continuity of $u$ we can find, for any $K>0$, a number $\de >0$ such that
$u_0(x) \ge \de$ whenever $|x| \le K+\sqrt T$.

Next, from the lower bound in \eqref{eqdiffeqtwosidedbound} (even its local version) we get
\[
u_t(x) \ge C_1 \int_{|x-\xi| \le \sqrt t} G_{\si_1}(t, x-\xi) u_0 (\xi) \, d\xi.
\]
If $|x| \le K$ we have $|\xi| \le K+\sqrt T$ in the integral. Consequently, for  $|x| \le K$,
\[
u_t(x) \ge \de C_1 \int_{|y| \le \sqrt t} G_{\si_1}(t, y) \, dy
\]
\[
\ge  \de C_1 (2\pi t \si^2)^{-d/2} \exp \{ -\frac{1}{2 \si_1^2} \} \int_{|y| \le \sqrt t} \, dy
=\de C_1 (2\pi \si^2)^{-d/2} \exp \{ -\frac{1}{2 \si_1^2} \}V_d,
\]
where $V_d$ is the volume of the unit ball in $R^d$.
This lower bound proves the proposition.
\end{proof}

Combining these two facts we arrive at the following main result of this section.

\begin{prop}
\label{proplocalizemed}
 Consider equation \eqref{eqdiffeq} under the assumptions
 of Proposition \ref{proptwosidedheat} (i) and the additional condition that the solution preserves the $L_1$-norm
 (like equation \eqref{eqSDEmedprelimdual} that we are mostly interested in).
 Let also arbitrary $\al=(\al_1,\cdots, \al_d)$ with $\al_j\in (0,1)$
 and $T>0$ be given.

 (i) For any $u_0 \in C(\R^d)\cup L_1(\R^d)$ with $\|u_0\|_L=1$, there exist $K,\de, \ep>0$
 such that, for any $u\in U_{m,M,(0,T]}(u_0)$,
 \begin{equation}
\label{eq1proplocalizemed}
 \inf \{u(x): \max_j |x_j| \le 2K \} \ge \de,
 \end{equation}
\begin{equation}
\label{eq2proplocalizemed}
 \min_j |Q^j_{\al} (u)| \le K
 \end{equation}
 and
 \begin{equation}
\label{eq3proplocalizemed}
\int_{\{x\in \R^d: \, \min_j|x_j|\ge K\}} |u(x) | \, dx  \le \ep.
\end{equation}
The set of functions satisfying conditions \eqref{eq1proplocalizemed} - \eqref{eq3proplocalizemed}
is convex.

(ii) There exist constants $K, \de >0$ such that, for any $\xi=(\xi_1, \cdots , \xi_d) \in \R^d$ and
$u\in U_{m,M,t}(\de_{\xi})$, $t\in (0, T]$,
  \begin{equation}
\label{eq4proplocalizemed}
 \inf \{u(x): \max_j |x_j-\xi_j| \le 2K \} \ge \de t^{-d/2},
 \end{equation}
\begin{equation}
\label{eq5proplocalizemed}
 \min_j |Q^j_{\al} (\de_{\xi})-\xi_j| \le K \sqrt t
 \end{equation}
 and
 \begin{equation}
\label{eq6proplocalizemed}
\int_{\{x\in \R^d: \, \min_j|x_j-\xi_j|\ge K \sqrt t\}} |u(x) | \, dx  \le \ep.
\end{equation}
  \end{prop}

  \begin{proof}
(i)  Recall that $Q_{\al}$ are defined by \eqref{eqdefddimquant}.
Let us pick up an $\ep$ such that
\[
\ep < \min (\al_1, \cdots, \al_d, 1-\al_1, \cdots, 1-\al_d).
\]
Applying Proposition  \ref{propcorolupperbound} we can find $K$ such that
\eqref{eq1propcorolupperbound} or \eqref{eq3proplocalizemed} holds for all $u\in  U_{m,M,[0,T]}(u_0)$ implying
 \eqref{eq2proplocalizemed}. Applying Proposition \ref{propcorollowerbound} yields
 \eqref{eq1proplocalizemed} for some $\de >0$. The last statement is obvious.

(ii) Notice that $u\in U_{m,M,t}(\de_{\xi})$ is in fact the heat kernel
\[
u=G(t,x,0,\xi)
\]
for the diffusion equation of a considered class.
Using the upper bound in \eqref{eqdiffeqtwosidedbound}, the free heat kernel \eqref{eqfreeheatkern}
and a well known (and easy to prove) estimate
\[
(2\pi t \si^2)^{-1/2} \int_K^{\infty} \exp \{ -\frac{x^2}{2t \si^2} \} dx \le C \exp \{ -\frac{K^2}{2t \si^2} \}
\]
for $t\in (0,T]$ with some constant $C$ depending on $T$ and $\si$ only, we obtain
\eqref{eq6proplocalizemed} for large enough $K$. This implies \eqref{eq5proplocalizemed}.
The lower bound in \eqref{eqdiffeqtwosidedbound} implies \eqref{eq4proplocalizemed}.
\end{proof}

\section{Sensitivity analysis for diffusion equations}

Here we shall exploit the smoothness of the heat kernels of diffusion equations to prove
 Lipschitz continuity, in $L_1$-norm, of the solutions to diffusion equations with respect to
functional parameters. The results may be known, but the author did not find an appropriate reference.

\begin{prop}
\label{propsensitivfordif1}
 Consider two equations \eqref{eqdiffeq}, specified by two families of operators $L_t^1, L_t^2$ with the coefficients
  $a_1,b_1,c_1$ and $a_2,b_2,c_2$ respectively, each satisfying
 the assumptions of Proposition \ref{proptwosidedheat}.

(i) Assume that the second order part of $L_t^1, L_t^2$ coincide, that is $a_1=a_2$.
  Then for any $u_0 \in L_1(\R^d)$, the solutions
 $u_t^1, u_t^2$ of the corresponding Cauchy problems satisfy the estimate
 \begin{equation}
\label{eq1propsensitivfordif1}
 \|u_t^1-u_t^2\|_L \le C \sqrt t \sup_{t,x} (|b_1(t,x)-b_2(t,x)|+|c_1(t,x)-c_2(t,x)|)\|u_0\|_L.
 \end{equation}

 (ii) Not assuming that $a_1=a_2$, but instead requiring that all coefficients $a_i,b_i,c_i$ be twice continuously differentiable
 in $x$ (with bounded derivatives), it follows that for any $u_0 \in H^2_1(\R^d)$
 the solutions
 $u_t^1, u_t^2$ of the corresponding Cauchy problems satisfy the estimate
 \begin{equation}
\label{eq1propsensitivfordif1}
 \|u_t^1-u_t^2\|_L \le C \sqrt t \sup_{t,x}
 (|a_1(t,x)-a_2(t,x)|+|b_1(t,x)-b_2(t,x)|+|c_1(t,x)-c_2(t,x)|)\|u_0\|_{H^2_1(\R^d)}.
 \end{equation}
  \end{prop}

  \begin{proof}
(i)  We apply the usual method for comparing semigroups or propagators. Denoting by $U_i^{t,s}$ the propagator solving equation \eqref{eqdiffeq} for $L^i$, that is
  \[
  U_i^{t,s} f(x)=\int G_i(t,x, s, \xi) f(\xi) d\xi, \quad i=1,2,
  \]
  with $G_i$ the heat kernel of the diffusion generated by $L_i$,
we see that they satisfy the identity
\begin{equation}
\label{eq2propsensitivfordif1}
   U_1^{t,0}-U_2^{t,0}=\int_0^t \frac{d}{ds} U_2^{t,s}U_1^{s,0} \, ds
   =\int_0^t U_2^{t,s}(L_s^1-L_s^2)U_1^{s,0}.
\end{equation}
From \eqref{eqdiffeqheatkernelderiv} it follows that, for any $u_0\in L_1(\R^d)$,
\[
\|\nabla U_1^{s,0} u_0 \|_L \le C s^{-1/2} \| u_0\|_L
\]
with a constant $C$. Hence, from \eqref{eq2propsensitivfordif1} and using the assumption
 that the difference $L_s^1-L_s^2$ is a first order operator we conclude
that
\[
\|( U_1^{t,0}-U_2^{t,0})u_0\|_L \le C \| u_0\|_L \int_0^t s^{-1/2} ds=2Ct^{1/2}\| u_0\|_L,
\]
as required.

(ii) It follows by the same identity \eqref{eq2propsensitivfordif1} combined with the standard fact that the propagators (resolving operators) of uniformly  elliptic diffusion equations
are bounded operators in $H^2_1(\R^d)$ (here the assumption of the existence of the second derivatives
 of the coefficients is needed), and with the observation that
\[
|(L_s^1-L_s^2)f(x)|_L \le \sup_{t,x}(|a_1(t,x)-a_2(t,x)|+|b_1(t,x)-b_2(t,x)|+|c_1(t,x)-c_2(t,x)|)\|f\|_{H^2_1(\R^d)}.
\]
\end{proof}

\section{Proof of the Theorems}

{\it Proof of Theorems \ref{thmediannonlin1} and \ref{thmediannonlin2}.}

Let $d=1$. Let us pick up a $u_0$ and then choose $K,\de, \ep>0$ from Proposition \ref{proplocalizemed}.
  Let $u_1,u_2 \in  U_{m,M,[0,T]}(u_0)$. Then the function $\om (h)=Q_{\al}(u_1+h(u_2-u_1))$ satisfies the equation
  \[
  \int_{-\infty}^{\om (h)} [u_1(x)+h(u_2(x)-u_1(x))]\, dx =\al.
  \]
  Differentiating with respect to $h$ yields
  \[
  \om'(h)=- [u_1(\om (h))+h(u_2(\om (h))-u_1(\om (h)))]^{-1} \int_{-\infty}^{\om(h)} (u_2(x)-u_1(x))\, dx.
  \]
  By Proposition \ref{proplocalizemed} we deduce the following basic estimate
   \begin{equation}
\label{eqestimdermedonedim}
  |Q_{\al}(u_2)-Q_{\al}(u_1)|=|\om(1)-\om(0)|=|\int_0^1 \om' (h) \, dh| \le \frac{1}{\de} \|u_1-u_2\|_L.
  \end{equation}

  Similarly, for arbitrary $d$, we get for  $\om_j (h)=Q_{\al}^j(u_1+h(u_2-u_1))$ the equation
     \begin{equation}
\label{eqestimdermed0}
  \om'_j(h)\int_{\R^{d-1}} (u_1+h(u_2-u_1))(x)|_{x_j=\om_j(h)} \prod_{k\neq j} dx_k
  =\int_{\{x \in \R^d: \, x_j\le \om_j(h)\} } (u_2(x)-u_1(x))\, dx,
  \end{equation}
  implying, by \eqref{eq1proplocalizemed} and \eqref{eq2proplocalizemed}, the estimate
\begin{equation}
\label{eqestimdermed}
|Q^j_{\al}(u_2)-Q^j_{\al}(u_1)| \le \frac{1}{\de K^{d-1}} \|u_1-u_2\|_L
\end{equation}
for each $j$.

Next, let $C([0,T], \R^d)$ denote the Banach space of $R^d$-valued continuous functions on $[0,T]$ with the usual norm
$\|\om_.\|=\sup_t |\om_t|$, and let $\CC (K)$ be the cube of side $2K$ centered at the origin.
For any $\om_0 \in \CC (K)$ let $C_{\om_0}([0,T], \CC(K))$ denote the convex subset of $C([0,T], \R^d)$ consisting of curves with $\om_0$ given and with values in $\CC (K)$. Let
\[
\om_0=Q_{\al}(u_0).
\]
 For a given curve $\om_. \in C_{\om_0}([0,T], \CC (K))$ let $u_t[\om_.](x)$ denote the solution at time $t$ of equation \eqref{eqSDEmedprelimdual} with the initial data $u_0(x)$.
 Let us define
 \[
 \hat \om_t=Q_{\al} (u_t[\om_.]).
 \]
 By the strong continuity of the propagators solving equation  \eqref{eqSDEmedprelimdual} in $L_1$ and by
 \eqref{eqestimdermed} we conclude that $\hat \om_t$ depends continuously on $t$. Consequently, applying Proposition
 \ref{proplocalizemed} we deduce that the mapping $\om_. \mapsto \hat \om_. $ is a mapping from $C_{\om_0}([0,T], \CC(K))$
 to itself. It is clear that the functions $u_t$ solve equation \eqref{eqSDEmeddensity}
 if and only if $\om_t= Q_{\al} (u_t)$ is a fixed point of this mapping. Thus well posedness of equation \eqref{eqSDEmeddensity}
 is reduced to the problem of uniqueness and existence of this fixed point.

 Let $\om_.^1, \om_.^2$ be two curves in $C_{\om_0}([0,T], \CC(K))$ and $\hat \om_.^1, \hat \om_.^2$ their respective images.
 By \eqref{eqestimdermed},
 \[
 |\hat \om_t^1-\hat \om_t^2|\le \frac{1}{\de K^{d-1}} \|u_t[\om_.^1]-u_t[\om_.^2]\|_L.
 \]

Moreover, from Proposition \ref{propsensitivfordif1}, we get under the assumptions of Theorems \ref{thmediannonlin1}
or \ref{thmediannonlin2} the estimates

 \[
 \|u_t[\om_.^1]-u_t[\om_.^2]\|_L \le \sqrt t C \ka \sup_{t\in [0,T]}|\om_t^1-\om_t^2| \|u_0\|_L
 \]
 or
   \[
 \|u_t[\om_.^1]-u_t[\om_.^2]\|_L \le \sqrt t C \ka \sup_{t\in [0,T]}|\om_t^1-\om_t^2| \|u_0\|_{H^2_1(\R^d)}
 \]
 respectively.
 In both cases this implies
 \begin{equation}
\label{eqcontractestimmed}
 |\hat \om_t^1-\hat \om_t^2|\le \sqrt t C \sup_{s \in (0,t]}|\om_s^1-\om_s^2|
 \end{equation}
 with another constant $C$ (depending on $u_0$). Consequently,  the mapping $\om_. \mapsto \hat \om_. $ is a contraction for small enough $T$. This implies the required well-posedness for small $T$. Global result follows by the usual iteration procedure
 completing the proof of Theorems \ref{thmediannonlin1} and \ref{thmediannonlin2}.

{\it Proof of Theorem \ref{thmediannonlin3}.}
Since the heat kernel (solution with the Dirac initial data) $G(t,x,0,\xi)$ is a positive continuous function for all $t>0$,
and by Theorems  \ref{thmediannonlin1} and \ref{thmediannonlin2}, it is enough to show the required well-posedness only
for small times. Next, by \eqref{eqestimdermed0}, \eqref{eqestimdermed} and the estimates
\eqref{eq4proplocalizemed}, \eqref{eq5proplocalizemed},  we find
\begin{equation}
\label{eqestimdermedGreen}
|Q^j_{\al}(u^2_t)-Q^j_{\al}(u^1_t)| \le \frac{1}{\de K^{d-1}} t^{d/2} t^{-(d-1)/2} \|u^1_t-u^2_t\|_L
\le \frac{1}{\de K^{d-1}} \sqrt t \|u^1_t-u^2_t\|_L
\end{equation}
for each coordinate $j$ and
\[
u^i_t(x)=G_i(t,x,0,\xi), \quad i=1,2,
\]
the heat kernels of two equations \eqref{eqSDEmedprelimdual} from the class specified by Proposition
\ref{proptwosidedheat} (i).

Next, under the assumptions of Theorem \ref{thmediannonlin1}, let us follow the same strategy by looking at the fixed points
of the mapping $\om_. \mapsto \hat \om_.$ from  $C_{\xi}([0,T], \CC_{\xi}(K \sqrt T))$ to itself,
where $\CC_{\xi}(K \sqrt T)$ is the cube of radius
$2K \sqrt T$ centered at $\xi$. As in Proposition \ref{propsensitivfordif1} (and noting that the norm of $\de_{\xi}$ is one)
we get for heat kernels arising from $\om^1, \om^2$ the estimate
  \[
 \|u_t[\om_.^1]-u_t[\om_.^2]\|_L \le \sqrt t C \ka \sup_{s\in [0,t]}|\om_s^1-\om_s^2|.
 \]
 Combining this with the previous estimate leads again to \eqref{eqcontractestimmed} implying the contraction
  property of the mapping $\om \mapsto \hat \om$ and hence the required well posedness.

 Let us turn to the assumptions of Theorem \ref{thmediannonlin2}.
 It will be more convenient now to work not with the quantiles $\om_t$, but with their transformations
 \[
 \eta_t=t^{-1/2}(\om_t-\xi).
 \]
 Let us look for a fixed point of the mapping
 \[
 \eta_. \mapsto \om_. \mapsto u_. \mapsto \hat \om_. \mapsto \hat \eta_.
 \]
 in the set $C((0,T], \CC(K))$ of continuous functions $(0,T]\mapsto  \CC(K)$. To see that $\eta \mapsto \hat \eta$ is in fact a mapping in this space we note that $\hat \eta_t \in \CC(K)$ for all $t$ (by Proposition \ref{proplocalizemed} (ii)) and that $\hat \eta_t$ is continuous for $t>0$ (by the proof of Theorem \ref{thmediannonlin3}). It remains to obtain the contraction property.

 As in the proof of Proposition \ref{propsensitivfordif1} (ii) (and using the fact that the second derivative of $G (t,x,0,\xi)$
 with respect to $x$ has singularity of order $1/t$) we obtain
 \[
 \|u_t[\om_.^1]-u_t[\om_.^2]\|_L \le C \ka \int_0^t \frac{1}{s} |\om_s^1-\om_s^2| \, ds,
 \]
 and hence
  \[
 \|u_t[\om_.^1(\eta_.^1)]-u_t[\om_.^2(\eta_.^2]\|_L \le C \ka \sqrt t \sup_{s\in (0,t]}|\eta_s^1-\eta_s^2|.
 \]
Combining with \eqref{eqestimdermedGreen} we get
 \begin{equation}
\label{eqcontractestimmedGreen}
 |\hat \eta_t^1-\hat \eta_t^2|\le \sqrt t C \sup_{s \in (0,t]} |\eta_t^1-\eta_t^2|,
 \end{equation}
 which again implies the required contraction property.

{\it Proof of Theorem \ref{thmediannonlinstable}.}

It goes along the same line. We just need the smoothing properties of equation
\eqref{eqstablemedprelimdensity}. They are developed in \cite{Ko00}, \cite{Ko00book} and \cite{Ko07}, see also Chapter 7 of
\cite{KoMarbook}. It is shown (see Theorems 3.1, 3.2 of \cite{Ko00} for time-homogeneous equations and
 Theorems 4.1, 4.2 of \cite{Ko07} for a
 straightforward extension to time non-homogeneous case) that
the transition probabilities (Green functions) $G_{\al}$ of processes generated by operators \eqref{eqstablemedprelimgen} satisfy the two sided estimate
\begin{equation}
\label{eqstableeqtwosidedbound}
C_1 S_{\al}(t-s, x-\xi; a, b)
\le G_{\al}(t,x, s, \xi) \le C_2 S_{\al}(t-s, x-\xi; a, b), \quad 0<s,t<T
\end{equation}
with constants $C_1, C_2$ and $S_{\al}$ denoting the stable density, i.e. the transition probability of the process
generated by \eqref{eqstablemedprelimgen} with constant $a,b$ ($C_1,C_2,a,b$ depend only on the bounds for
$a(x), a^{-1}(x), b(t,x, \om)$ and their derivatives in $x$), and the derivative of $G_{\al}$ with respect to $x$ (or $\xi$)
has the bound
\begin{equation}
\label{eqstablekernelderbound}
\frac{\pa G_{\al}}{\pa x}(t,x, s, \xi)
\le C(t-s)^{-1/\al} G_{\al}(t,x, s, \xi), \quad 0<s,t<T
\end{equation}
(see formula (58) of \cite{Ko07}).

With these estimates and the standard estimates for $S_{\al}$ (implying that $S_{\al}$ is of order $t^{-d/\al}$ for a ball of
radius $t^{1/\al}$ around the diagonal, see e.g. books \cite{Ko00} or \cite{KoMarbook})
the proof of Theorem \ref{thmediannonlinstable} and all intermediate results repeats literally the proof
of Theorem \ref{thmediannonlin1}.

{\bf Acknowledgements}. The author is grateful to Tom Kurtz for suggesting him this nice problem
and to Sigurd Assing and Astrid Hilbert for fruitful discussion.

\end{document}